\documentclass[oneside,11pt]{amsart}
\usepackage[pagebackref,breaklinks,unicode]{hyperref}
\hypersetup{pdftitle = Recognizing convergence actions}
\title[Recognizing convergence actions]{Recognizing convergence actions as proper isometric actions on hyperbolic spaces}

\usepackage[margin=2.99cm]{geometry}

\author{Lihuang Ding}
\address{Department of Mathematics and Statistics, Queen's University, Canada}
\email{lihuang.ding@queensu.ca}

\author{Abdul Zalloum}
\address{Department of Mathematics and Statistics, Queen's University, Canada}
\email{az32@queensu.ca}

\usepackage{amsmath,amssymb,amsthm, mathrsfs, mathtools}
\usepackage{enumitem} \setlist{nosep} 
\usepackage[utf8]{inputenc}
\usepackage[english]{babel}
\usepackage[pagebackref,breaklinks,unicode]{hyperref}
\usepackage{float}
\usepackage{comment}
\usepackage{csquotes}
\usepackage[font=small,justification=centering]{caption}
\usepackage{subcaption}
\usepackage[dvipsnames]{xcolor}
\usepackage{mathabx} 
\usepackage{xfrac} 
\usepackage[normalem]{ulem} 
\usepackage{dirtytalk}
\usepackage{cancel}
\usepackage{marginnote}

\newtheorem{theorem}{Theorem}[section]
\newtheorem{proposition}[theorem]{Proposition}
\newtheorem{question}[theorem]{Question}

\newtheorem{corollary}[theorem]{Corollary}
\newtheorem{lemma}[theorem]{Lemma}

\theoremstyle{definition}

\newtheorem{remark}[theorem]{Remark}

\newcounter{shcount}
 
\newcommand*{\ubsh}[1]{\theoremstyle{definition}\newtheorem*{subhead\theshcount}{#1}
    \begin{subhead\theshcount}} 
\newcommand*{\uesh}{\end{subhead\theshcount}\stepcounter{shcount}} 

\newcounter{claimcount}

\newenvironment{claim*}[1]{\par\vspace{2mm}\noindent
    \underline{Claim:}\hspace{2mm}#1}{}

\newcounter{enumlabelcount}
\makeatletter\newcommand\enumlabel[1][]{\item[#1]
    \refstepcounter{enumlabelcount}\def\@currentlabel{#1}}\makeatother

\renewcommand*{\backrefalt}[4]{\ifcase #1 (Not cited).\or (Cited p.~#2).\else (Cited pp.~#2).\fi} 

\makeatletter\def\subsection{\@startsection{subsection}{1}\z@{.7\linespacing\@plus\linespacing}
    {.5\linespacing}{\normalfont\scshape\centering}}\makeatother 
\makeatletter\def\part{\@startsection{part}{1}\z@{.7\linespacing\@plus\linespacing}
    {.5\linespacing}{\normalfont\large\scshape\centering}}\makeatother 

\newcommand{\eps}{\varepsilon}

\definecolor{harrycomment}{rgb}{0.6,0,0.4}

\begin{document}

\maketitle

\begin{abstract}
Relying merely on B.I. (a.k.a. Bowditch's intelligence), we show that whenever a group $G$ admits a convergence action on a compact perfect metrizable space $M$, there is a hyperbolic space $X$ along with a proper isometric action of $G$ on $X$ inducing a $G$-equivariant homeomorphism $f:M \rightarrow \partial X.$ This answers a question by Misha Kapovich.
\end{abstract}

\section{Introduction} 

An action of a group $G$ on a topological space $M$ by homeomorphisms is a \emph{convergence action} if the induced action on the space of distinct triples $\theta_3(M)$ is properly discontinuous. A convergence action is \emph{uniform} if the action on $\theta_3(M)$ is cocompact. In \cite[Thm.~0.1]{Bowditch1998}, Bowditch showed that a uniform convergence action on a perfect compactum (a compact metrizable space with no isolated points) yields a proper cobounded action on a hyperbolic space $X$ (hence $G$ is hyperbolic). In \cite[Question~7]{Kapovich2007Boundaries}, Misha Kapovich asks whether replacing cocompactness on $\theta_3(M)$ with minimality on $M$ still yields a proper action on a hyperbolic space whose boundary is $G$-equivariantly identified with $M$. More precisely:

 \begin{question}\cite[Question~7]{Kapovich2007Boundaries} \label{question:Misha} Let $G$ be a group admitting a convergence action on a perfect compactum $M$ such that every $G$-orbit is dense. Does there exist a hyperbolic space $X$ with Gromov boundary $M$ so that the $G$-action on $M$ extends to a uniformly quasi-isometric quasi-action on $X$?
    
\end{question}

The following example illustrates Question \ref{question:Misha}. Let $G$ be an infinite normal subgroup of a hyperbolic group $K$ acting geometrically on a hyperbolic space $X$ (e.g., $G=\pi_1(S)$ is the fiber subgroup of $K=\pi_1(N)$, where $N$ is a closed hyperbolic $3$-manifold fibering over the circle with fiber $S$, and $X=\mathbb{H}^3$). The induced $G$-action on $M=\partial X$ is a minimal convergence action. Loosely speaking, Question \ref{question:Misha} can be interpreted as asking whether this action alone, without knowledge of $K$, suffices to coarsely recover $X$ and the proper $G$-action on it. The following answers Question \ref{question:Misha} affirmatively in a strong sense:

\begin{theorem} \label{thm:main} Let $G$ be a group that admits a convergence action on a perfect compactum $M$. There is a geodesic hyperbolic space $X$ and a proper action of $G$ on $X$ by isometries inducing a $G$-equivariant homeomorphism $f: M \rightarrow \partial X$. 
\end{theorem}

Theorem \ref{thm:main} complements Bowditch's characterization \cite[Thm.~0.1]{Bowditch1998} by showing that the cocompactness assumption on $\theta_3(M)$ can be dropped if one only seeks a proper action on a hyperbolic space with boundary $M$, rather than a proper, cobounded one. Furthermore, observe that the minimality assumption of the $G$-action on $M$ in Question \ref{question:Misha} is not needed in Theorem \ref{thm:main}. Indeed, the theorem applies even when $G=\{e\}$ and in this case, it yields:

\begin{corollary} \label{cor:identity} A perfect compactum is the Gromov boundary of a hyperbolic space. More precisely, if $M$ is a perfect compactum, then there is a hyperbolic metric $\rho$ on the space of distinct triples $\theta_3(M)$ with $\partial (\theta_3(M), \rho) \cong M.$
    
\end{corollary}

Bonk and Schramm \cite{BonkSchramm2000} showed that every complete bounded metric space $M$ is the Gromov boundary of a hyperbolic space. More precisely, they equip $\operatorname{Cone}(M):=M\times(0,D]$, where $D=\operatorname{diam}(M)$, with the hyperbolic metric $\rho((x,t),(y,s))=2\log\left(\frac{d(x,y)+\max\{t,s\}}{\sqrt{ts}}\right),$ whose Gromov boundary is $M$. Theorem \ref{thm:main} gives an equivariant strengthening of their boundary realization result. Indeed, for any perfect compactum $M$ and any finite subgroup $G<\operatorname{Homeo}(M)$, or any infinite cyclic subgroup $G=\langle g\rangle<\operatorname{Homeo}(M)$ acting with north--south dynamics, there exists a geodesic hyperbolic space $X$ with $\partial X\cong M$ equivariantly and $G<\operatorname{Isom}(X)$ acting properly.

We remark that recent work of Chaodong Yang and Wenyuan Yang reiterates Question \ref{question:Misha} in \cite[Remark~1.5]{YangYang2026}, whose results, combined with Theorem \ref{thm:main}, yield a positive answer to \cite[Q.~1.4]{YangYang2026}: every discrete group admitting a convergence action on a compact metrizable space with countably many non-conical limit points admits a geometrically finite action on a hyperbolic space. More recently, Chaodong Yang independently obtained a positive answer to this question using different methods \cite{Yang2026}.

The following is a natural question (we thank Kathryn Mann for raising it).

\begin{question}[Cannon--Thurston maps]\label{question:cannon_Thurston} Let $G$ be a hyperbolic group admitting a minimal convergence action on a perfect compactum $M$. By Theorem \ref{thm:main}, $G$ acts properly on a hyperbolic space $X$ with $\partial X\cong M$. Under what conditions on $G \curvearrowright M $ $($or on $G \curvearrowright X)$ does an orbit map $G\to X$ extend continuously and equivariantly to a map $\partial G\to\partial X$?
    
\end{question}

Without additional assumptions, the answer to Question \ref{question:cannon_Thurston} is negative. Indeed, Baker and Riley \cite{BakerRiley2013} constructed a free subgroup $G=\mathbb{F}_3$ of a hyperbolic group $K$ whose inclusion $G\hookrightarrow K$ admits no continuous extension to the Gromov boundaries. Taking $M$ to be the limit set of $G$ in $\partial K$ yields a minimal convergence action of $G$ on a perfect compactum $M$ admitting no continuous extension $\partial G\to M$.


\section{A comment on the proof} Much of the groundwork for Theorem \ref{thm:main} was laid by Bowditch \cite{Bowditch1998}, some of whose ideas were independently rediscovered in \cite{petytzalloum:constructing}. We explain how to build on his arguments to obtain our result. Let $M$ be a perfect compactum and let $\theta_3(M) \subset M^3$ denote the subspace of distinct triples (i.e., no two coordinates are equal). An \emph{annulus} in $M$ is a pair $A=(A^-,A^+)$ where $A^-,A^+ \subset M$ are non-empty, closed, disjoint and satisfy $M \setminus (A^- \cup A^+) \neq\emptyset$. Let $A=(A^-,A^+), B=(B^-,B^+)$ be two annuli; we say $A,B$ \emph{cross} if $A^+ \cap B^+, A^+ \cap B^-, A^- \cap B^+, A^- \cap B^-$ are all non-empty. We say $A$ is \emph{nested} in $B$ (and write $A<B$) if $M \setminus B^- \subset A^+$ and a collection of annuli $c=\{A_1, \cdots, A_n\}$ is said to form a \emph{chain} if, possibly after reordering, we have $A_1<A_2< \cdots < A_n$. Two chains $c_1,c_2$ \emph{cross} if every annulus of $c_1$ crosses every annulus of $c_2.$ An annulus $A=(A^-,A^+)$ is said to \emph{separate} $Y,Z \subset M$ if $Y \subset A^-, Z \subset A^+$ or vice versa. A chain $c$ is said to separate $Y,Z\subset M$ if each $A \in c$ does. An annulus $A$ is said to separate two distinct triples $(x_1,x_2,x_3), (y_1,y_2,y_3) \in \theta_3(M)$ if $|\{x_1, x_2, x_3\} \cap A^-| \geq 2$ and $|\{y_1, y_2, y_3\} \cap A^+| \geq 2$ or vice versa. A chain $c$ is said to separate $(x_1,x_2,x_3), (y_1,y_2,y_3) \in \theta_3(M)$ if each annulus $A \in c$ does.  An \emph{annulus system} is simply a collection of annuli. Given an annulus system $\mathcal{A}$, define
\begin{center}
    $(Y|Z)_\mathcal{A}:= \text{sup}\{|c|: c \subset \mathcal{A} \text{ is a chain separating }Y,Z\} \in \mathbb{N} \cup \{0\} \cup \{\infty\}.$
\end{center}

To simplify notation, we use $(xy|zw)_{\mathcal{A}}$ to denote $(\{x,y\}|\{z,w\})_{\mathcal{A}}$. An annulus system $\mathcal{A}$ is \emph{discrete} if for any $x \neq y$ and any $z \neq w$, we have $(xy|zw)_{\mathcal{A}}<\infty$. A discrete annulus system is said to be \emph{$L$-separated} if whenever $c_1,c_2 \subset \mathcal{A}$ are chains that cross, we must have $\text{min}\{|c_1|, |c_2|\} <L.$ A \emph{separated} annulus system is an $L$-separated annulus system for some $L.$ In \cite[Prop. 6.5]{Bowditch1998}, Bowditch shows that each separated annulus system  $\mathcal{A}$ gives rise to a hyperbolic metric $\rho_\mathcal{A}$ on $\theta_3(M)$ (a fact that was rediscovered and generalized in \cite{petytzalloum:constructing}, see \cite[Thm K]{petytzalloum:constructing}) by defining 
$\rho_{\mathcal{A}}(x,y):=1+\sup\{|c|: c  \subset \mathcal{A} \text{ is a chain separating }x,y\}$ whenever $x,y \in \theta_3(M)$ are distinct and declaring $\rho_{\mathcal{A}}(x,x)=0$. Namely:

\begin{theorem}[\cite{Bowditch1998}]\label{thm:scattered_1} If $\mathcal{A}$ is a separated annulus system on a perfect compactum $M$, then the metric space $X_{\mathcal{A}}:=(\theta_3(M), \rho_\mathcal{A})$ is a roughly geodesic hyperbolic space. If $G<\text{Homeo}(M)$ preserves $\mathcal{A},$ then $G$ acts on $X_\mathcal{A}$ by isometries. 
\end{theorem}

For the rest of the article, we will use $X_\mathcal{A}$ to denote $(\theta_3(M), \rho_\mathcal{A}).$  Given a separated annulus system $\mathcal{A}$ on a perfect compactum $M$, Bowditch calls $\mathcal{A}$ \emph{compatible} with the topology on $M$ if for any distinct $x,y,z$ and a sequence $x_i \rightarrow x$, we have $(xx_i|yz)_\mathcal{A} \rightarrow \infty.$ If for any distinct $x,y,z \in M$, we have $(xy|z)_{\mathcal{A}}=\infty$, then $\mathcal{A}$ is compatible with the topology on $M$ as shown in \cite[Lem. 6.6]{Bowditch1998}. Combining this with \cite[Lem. 4.6, Prop. 4.7, Prop. 4.8]{Bowditch1998} gives:

\begin{theorem}[\cite{Bowditch1998}]\label{thm:scattered_2} Let $M$ be a perfect compactum and let $\mathcal{A}$ be a separated annulus system with $(xy|z)_{\mathcal{A}}=\infty$ for all distinct $x,y,z \in M$. There is a natural homeomorphism $f:M \rightarrow \partial X_\mathcal{A}$ and if $G<\text{Homeo}(M)$ preserves $\mathcal{A}$, then $f$ is $G$-equivariant.
    
\end{theorem}

The proof of the following lemma is straightforward; we include it in Section \ref{sec:Proof} for completeness.

\begin{lemma}\label{lem:proper_action} If $G$ admits a convergence action on a perfect compactum $M$ preserving a separated annulus system $\mathcal{A}$ with $(xy|z)_{\mathcal{A}}=\infty$ for all distinct $x,y,z \in M$, then the induced isometric action on $X_\mathcal{A}$ is proper.
    
\end{lemma}


In \cite[Prop.3.7]{Lang13}, Lang shows that a geodesic hyperbolic space $X$ admits an $\text{Isom}(X)$-equivariant isometric embedding $i$ into an \emph{injective} metric space $E(X)$ with a coarsely dense image (i.e., there is $C \geq 0$ such that each $y \in E(X)$ is within $C$ of $i(X)$). In \cite[Prop. A.2]{PSZ24}, it is observed that the above remains true when $X$ is a roughly geodesic hyperbolic space. Since injective spaces are geodesic, combining  \cite[Prop. A.2]{PSZ24} with Theorem \ref{thm:scattered_1}, Theorem \ref{thm:scattered_2}, and Lemma \ref{lem:proper_action} yields the following. 


\begin{theorem}\label{thm:summary} If $G$ admits a convergence action on a perfect compactum $M$ preserving a separated  annulus system $\mathcal{A}$ with $(xy|z)_\mathcal{A}=\infty$ for all distinct $x,y,z \in M,$ then $G$ acts properly by isometries on a geodesic hyperbolic space $X$ inducing a $G$-equivariant homeomorphism $f:M \rightarrow  \partial X$.
    
\end{theorem}

By Theorem \ref{thm:summary}, proving Theorem \ref{thm:main} reduces to constructing a $G$-invariant separated annulus system $\mathcal{A}$ with $(xy|z)_\mathcal{A}=\infty$ for all distinct $x,y,z\in M$. In fact, Theorem \ref{thm:good_annulus_system} constructs an annulus system that satisfies the stronger condition $(K|x)_\mathcal{A}=\infty$ for every compact $K\subset M$ and $x\notin K$. Our construction is inspired by Bowditch's \cite[Prop.~8.2]{Bowditch1998}, where, assuming only a convergence action on a perfect compactum $M$, he constructs a separated $G$-invariant annulus system $\mathcal{A}$ satisfying $(x|y)_\mathcal{A}>0$ for all distinct $x,y\in M$.  Bowditch then uses the additional assumption that every point of $M$ is conical to establish $(K|x)_\mathcal{A}=\infty$ for every compact $K\subset M$ and $x\notin K$. The key observation of this paper (Theorem \ref{thm:good_annulus_system}) is that one can construct a separated annulus system $\mathcal{A}$ with $(K|x)_\mathcal{A}=\infty$ (for all compact $K \subset M$ and $x \notin K$) without the extra assumption that every point of $M$ is conical.

\ubsh{Acknowledgements}\emergencystretch=2em We thank Thomas Barthelmé, Sami Douba, KyeongRo Kim, Kathryn Mann, Alessandro Sisto and Wenyuan Yang for fruitful discussions. We are particularly grateful to Sami Douba for bringing Misha Kapovich's question to our attention following a talk by the second author at the Max Planck Institute in Leipzig in July 2026.

\uesh


\section{Constructing the desired annulus system}\label{sec:Proof} 

We start by providing a proof to Lemma \ref{lem:proper_action}, which we restate here for completeness.

\begin{lemma} If $G$ admits a convergence action on a perfect compactum $M$ preserving a separated annulus system $\mathcal{A}$ with $(xy|z)_{\mathcal{A}}=\infty$ for all distinct $x,y,z \in M$, then the induced isometric action on $X_\mathcal{A}$ is proper.
    
\end{lemma}

\begin{proof} Since $M$ is compact, for any $(x,y,z) \in \theta_3(M)$ and any infinite sequence of distinct elements $\{g_i\} \subset G$, after passing to a subsequence, we have $g_ix \rightarrow a, \,y_i=g_iy \rightarrow b$ and $g_iz \rightarrow c.$ Since the action on $\theta_3(M)$ is properly discontinuous, the points $a,b,c$ can't all be distinct, say $b=c.$ It is possible that $b \in \{x,y,z\}$, so let $d_1,d_2 \in \{x,y,z\}$ be distinct points with $d_1 \neq b$ and $d_2 \neq b.$ By assumption, we have $(d_1d_2|b)_\mathcal{A}=\infty$, that is, for each integer $n \geq 1$, there is $m \geq n$ and a chain $$A_{n_1}=(A_{n_1}^-,A_{n_1}^+)<A_{n_2}=(A_{n_2}^-, A_{n_2}^+)< \cdots <A_{n_m}=(A_{n_m}^-, A_{n_m}^+),$$ with $\{d_1,d_2\} \subset A_{n_1}^-$ and $b \in A_{n_m}^+$. Since $M \setminus A_{n_m}^-$ is an open set containing $b$, we must have $g_iy, g_iz \in M \setminus A_{n_m}^-$ for all large enough $i.$ In particular, $A_{n_1}, \cdots A_{n_{m-1}}$ all separate $\{d_1,d_2\}$ from $\{y_i,z_i\}$ where $y_i=g_iy$ and $z_i=g_iz$. This shows that $(d_1d_2|y_iz_i)_\mathcal{A} \geq m-1$ for all large enough $i$. Since the above holds for any $n \geq 1$ and since $m \geq n,$ we have $(d_1d_2|y_iz_i)_\mathcal{A} \rightarrow \infty.$ This shows that $\rho_\mathcal{A}((x,y,z), g_i((x,y,z)) \rightarrow \infty$ completing the proof.
\end{proof}

Our proof of Theorem \ref{thm:main} makes substantial use of the following proposition.

\begin{proposition}[{\cite[Prop. 7.1]{Bowditch1998}}]\label{prop:convergece_characterize}The action of a group $G$ on a perfect compactum is a convergence action if and only if for any infinite sequence of distinct elements $\{g_n\} \subset G$, there is a subsequence $g_{n_k}$ and points $a,b \in M$ such that $g_{n_k}|_{M \setminus \{b\}} \rightarrow a$ uniformly on compact sets.
\end{proposition}

For the rest of this section, \textbf{every statement assumes that $G$ admits a convergence action on a perfect compactum $M$}. We will also fix once and for all a metric $d$ on $M$ inducing its topology. The following remark is crucial.
\begin{remark}\label{rem:key_remark} It immediately follows from Proposition \ref{prop:convergece_characterize} that if $G$ admits a convergence action on a perfect compactum $M$, then for any annulus $A=(A^-, A^+)$ and any infinite sequence of distinct elements $g_i$, after passing to a subsequence, we have $$\text{min}\{\text{diam} (g_iA^-), \text{diam}(g_iA^+)\}\rightarrow 0. $$ 
\end{remark}

The following is an immediate corollary of Remark \ref{rem:key_remark}
    
\begin{corollary}\label{cor:Bowditch's Key} For any finite collection of annuli $F=\{A_1, \cdots A_n\}$, the annulus system $\mathcal{F}:=G \cdot \{A_1, \cdots A_n\}$ is discrete.

\end{corollary}

\begin{proof} This can be extracted from \cite{Bowditch1998} but we provide the proof for completeness. Let $x,y,z,w \in M$ with $x \neq y$ and $z \neq w$. If there are infinitely many annuli in $\mathcal{F}$ separating $\{x,y\}$ from $\{z,w\}$, then, after passing to a subsequence, there is $A \in \{A_1, \cdots A_n\}$ and infinitely many $g_i \in G$ such that $g_iA$ separates $\{x,y\}$ from $\{z,w\}$. After passing to another subsequence, we get $\text{min}\{\text{diam}(g_iA^-), \text{diam}(g_iA^+)\} \rightarrow 0$ which is a contradiction to the fact that $x \neq y$ and $z \neq w$.
    
\end{proof}

The following lemma is unnecessary for our main theorem and can be completely skipped. However, it is a particularly powerful consequence of \cite{Bowditch1998} and highlights the power of \cite[Prop. 7.1]{Bowditch1998}. 

\begin{lemma} For any finite collection of annuli $F=\{A_1, \cdots A_n\}$, there is $L=L(F,G,M)$ such that the annulus system $\mathcal{F}=G \cdot F$ is $L$-separated. In particular, each finite collection of annuli $F$ gives rise to a hyperbolic metric $X_{\mathcal{F}}$ and an isometric action of $G$ on $X_{\mathcal{F}}$.
    
\end{lemma}

\begin{proof} By Corollary \ref{cor:Bowditch's Key}, the annulus system $\mathcal{F}$ is discrete. By equivariance, if each $A_i$ is crossed by at most $L_i$ members of $\mathcal{F}$, then we are done by taking $L=1+\text{max}_{1 \leq i \leq n}\{L_i\}$. Otherwise, there is some $A_i \in F$, say $A_1,$ which is crossed by infinite many $B_i \in \mathcal{F}$. Since $F$ is finite, after passing to a subsequence, we may assume $B_i=g_iA$ for some $A \in \mathcal{F}$. Since $\text{min}\{ \text{diam} (g_iA^-), \text{diam} (g_iA^+)\} \rightarrow 0$ and $A_1^+,A_1^-$ are closed, we get $A_1^+ \cap A_1^- \neq \emptyset$ which is a contradiction. 
    
\end{proof}

Given an annulus $A=(A^-,A^+)$, we define the \emph{width} of $A$ by $W(A)=d(A^-,A^+)$ and define the \emph{depth} of $A$ by $D(A):= \text{min}\{\text{diam}(A^-), \text{diam}(A^+)\}.$

\begin{lemma}\label{lemma:finiteness} For any $\epsilon>0$ we have $|\{g \in G: D(gA)\geq \epsilon\}|<\infty$. Further, if $|A^-| \geq2$ and $|A^+| \geq 2$, then $|\{g \in G: W(gA)\geq\epsilon\}|< \infty.$

\end{lemma}

\begin{proof} The first part is immediate by Remark \ref{rem:key_remark}. For the second part, if $g_i$ is an infinite sequence, then there is a subsequence $g_{i_k}$ and $a,b \in M$ such that $g_{i_k}|_{M \setminus \{b\}} \rightarrow a.$ Since $|A^-| \geq2$ and $|A^+| \geq 2$, we can choose $a^+ \in A^+,a^- \in A^-$ with $b \notin \{a^+,a^-\}$, and hence $d(g_{i_k}a^-, g_{i_k}a^+)) \rightarrow 0$. This completes the proof.
    
\end{proof}

Given $x \in M$ and $0<\epsilon_1< \epsilon_2<\text{diam}(M)$, we define the annulus $$A(x, \epsilon_1, \epsilon_2):=(\overline{B}(x,\epsilon_1), M \setminus B(x, \epsilon_2)),$$ where $B(x, \epsilon)$ is the ball of radius $\eps$ about $x.$ Observe that $W(A(x, \epsilon_1, \epsilon_2)) \geq \epsilon_2-\epsilon_1.$ The following lemma provides us with the annuli needed to achieve Theorem \ref{thm:main}.

\begin{lemma}\label{lemma:good_annulus} For each $x \in M$ and positive integer $n  > \frac{1}{\text{diam}(M)}$, there is $\epsilon=\epsilon(x,n)$ and an annulus $A(x,\epsilon,\frac{1}{n})$ such that: 

\begin{enumerate}

\item $W(A(x,\epsilon, \frac{1}{n}))> \frac{1}{n+1}$, and
\item $D(gA(x,\epsilon, \frac{1}{n})) < \frac{1}{n+2}$ for all $g \in G$.

\end{enumerate}

\end{lemma}

\begin{proof} Let $x \in M$ and $n >\frac{1}{\text{diam}(M)}$. Choose some $\delta< \frac{1}{n(n+1)}$ and consider the annulus $A=A(x,\delta, \frac{1}{n}).$  By Lemma \ref{lemma:finiteness}, there is only finitely many $g_1, \cdots g_m \in G$ such that $D(g_iA) \geq \frac{1}{n+2}.$ Since each $g_i$ is continuous, we may choose $0<\epsilon<\delta$ such that $g_i B(x,\epsilon) \subset B(g_ix, \frac{1}{3(n+2)})$ for all $i \in \{1 \cdots m\}.$ In particular, $\text{diam}(g_iB(x,\epsilon)) < \frac{2}{3(n+2)}$ which gives us that $\text{diam}(g_i\overline{B}(x,\epsilon)) \leq \frac{2}{3(n+2)}<\frac{1}{n+2}$. This shows that $A(x, \eps, \frac{1}{n})$ satisfies the second conclusion of the lemma. Since $\epsilon<\delta< \frac{1}{n(n+1)}$, the annulus $A(x, \eps, \frac{1}{n})$ satisfies the first conclusion as well which completes the proof.

\end{proof}

\begin{theorem}\label{thm:good_annulus_system} Let $G$ admit a convergence action on a perfect compactum $M$. There is a $G$-invariant annulus system $\mathcal{A}$ such that:

\begin{enumerate}
    \item $\mathcal{A}$ is 1-separated, and

    \item For any compact $K$ and $x \in M \setminus K$, we have $(x|K)_\mathcal{A}=\infty.$
\end{enumerate}
    
\end{theorem}


\begin{proof}\emergencystretch=2em  Given a positive integer $n >\frac{1}{\text{diam}(M)}$ and $x \in M$, consider $A(x, \epsilon(x,n),\frac{1}{n})=(\overline{B}(x,\epsilon(x,n)), M \setminus B(x, \frac{1}{n}))$ which is the annulus produced by Lemma \ref{lemma:good_annulus}. For a given $n,$ the collection $\{B(x, \epsilon(x, n))| x \in M\}$ covers $M$ and hence we get a finite subcover $\{B(x_i, \epsilon_i)\}_{i=1}^{m_n}$. Consider their respective collection of annuli given by $F_n=\{A(x_i, \epsilon_i, \frac{1}{n})\}_{i=1}^{m_n}$. We claim that the annulus system $$\mathcal{A}:=\bigcup_{n \in \mathbb{N}} G \cdot F_n$$
satisfies the conclusion of the theorem. 

To show $\mathcal{A}$ is $1$-separated, we need to first show $\mathcal{A}$ is discrete and so let $x,y,z,w \in M$ be with $x \neq y$ and $z \neq w$. We claim that there is $k$ such that every annulus separating $\{x,y\}$ from $\{z,w\}$ is in $\bigcup_{1 \leq i \leq k} G\cdot F_i$ in which case we are done by Corollary \ref{cor:Bowditch's Key}. To see this, observe that if $A$ separates $\{x,y\}, \{z,w\}$, then $D(A) \geq d(x,y)$ or $D(A) \geq d(z,w).$ In particular, $D(A) \geq \text{min}\{d(x,y), d(z,w)\}$.  By item 2 of Lemma \ref{lemma:good_annulus}, we have $D(A) \leq \frac{1}{n+2}$ for any $A \in G \cdot F_n$ and so $\frac{1}{n+2} \geq D(A) \geq \text{min}\{d(x,y), d(z,w)\}$. This bounds $n$ from above, that is, there is $k$ such that every annulus $A$ separating $\{x,y\}$ from $\{z,w\}$ is in $G \cdot F_i$ where $i \leq k$.

Now we show that no two annuli $A,B \in \mathcal{A}$ cross which implies that $\mathcal{A}$ is 1-separated. Assume for the sake of contradiction that some $A,B \in \mathcal{A}$ cross. In particular, there exist distinct $x,y,z,w \in M$ such that $A$ separates $\{x,y\}$ from $\{z,w\}$ and $B$ separates $\{x,z\}$ from $\{y,w\}.$ By definition of $\mathcal{A}$, we must have  $A \in G \cdot F_m$ and $B \in G \cdot F_n$. Without loss of generality, assume $m \leq n$. By equivariance, we may also assume that $A \in F_m$. In particular, by Lemma \ref{lemma:good_annulus}, we have $W(A) > \frac{1}{m+1} $ which gives us that $d(x,z)> \frac{1}{m+1}$ and $d(w,z)> \frac{1}{m+1}.$ In particular, we must have $D(B)  \geq \frac{1}{m+1}.$ But since $B \in G \cdot F_n$ with $n \geq m$, Lemma \ref{lemma:good_annulus} tells us $D(B) < \frac{1}{n+2} \leq \frac{1}{m+2}$ which is a contradiction. This shows that $\mathcal{A}$ is 1-separated. For the second part, let $K$ be compact and $x \in M \setminus K.$ Choose $n_1$ large enough so that for any $y \in M$ with $x \in B(y,\frac{1}{n_1})$, the ball $B(y,\frac{1}{n_1})$ is disjoint from $K.$ The finite cover obtained at the $n_1$-stage yields a ball $B(y_1,\epsilon_1) \subset B(y_1,\frac{1}{n_1})$ containing $x$ and disjoint from $K.$ Hence, the annulus $A_1:=A(y_1, \epsilon_1, \frac{1}{n_1})$ separates $x$ from $K$, consequently, $(x|K)_{\mathcal{A}} \geq 1.$ Now, take $n_2$ large enough so that any ball $B(y, \frac{1}{n_2})$ containing $x$ is properly contained in $B(y_1,\epsilon_1)$. In particular, the finite cover obtained at the $n_2$-stage yields a ball $B(y_2, \epsilon_2) \subset B(y_2, \frac{1}{n_2})$ containing $x$ and disjoint from $M \setminus B(y_1,\epsilon_1)$. Hence, we have $A_1>A_2$ and each $A_i$ separates $x$ from $K$ which shows $(x|K)_{\mathcal{A}} \geq 2.$ Continuing inductively shows $(x|K)_\mathcal{A}=\infty.$

\end{proof}

\ubsh{AI Declaration} Artificial intelligence tools were used solely to locate relevant literature. They were not used to develop, verify, or write any of the mathematical arguments presented in this article.

\uesh

\bibliographystyle{alpha}
\bibliography{references}

\end{document}